\documentclass[12pt]{amsart}
\usepackage{times,amssymb,amsmath}
\usepackage{comment}
\usepackage[usenames, dvipsnames]{color}
\usepackage{hyperref}

\parskip=\smallskipamount

\newtheorem*{Thm*}{Theorem}
\newtheorem{Thm}{Theorem}
\newtheorem{Cor}{Corollary}
\newtheorem{Prop}{Proposition}
\newtheorem{Lemma}{Lemma}

\theoremstyle{definition}
\newtheorem{Defn}{Definition}
\newtheorem{Notation}{Notation}

\newtheorem{Remark}{Remark}
\newtheorem{Example}{Example}
\newtheorem{Construction}{Construction}

\newcommand{\abs}[1]{\left\vert#1\right\vert}

\newcommand{\set}[1]{\left\{#1\right\}}
\newcommand{\mf}[1]{\mathbb{#1}}

\newcommand{\mc}[1]{\mathcal{#1}}
\newcommand{\mb}[1]{\mathbf{#1}}
\newcommand{\ip}[2]{\left \langle{#1},{#2} \right \rangle}
\newcommand{\norm}[1]{\left \|{#1} \right \|}

\DeclareMathOperator{\Span}{\mathrm{Span}}

\DeclareMathOperator{\Inner}{\mathrm{Inner}}
\DeclareMathOperator{\Outer}{\mathrm{Outer}}

\title{Fock representation of free convolution powers}
\author{Michael Anshelevich, Jacob Mashburn}
\thanks{This work was supported in part by a Simons Foundation Collaboration Grant.}
\address{Department of Mathematics, Texas A\&M University, College Station, TX 77843-3368}
\email{manshel@tamu.edu, jacobmashburn16@tamu.edu}
\begin{document}

\begin{abstract}
Let $\mc{B}$ be a star-algebra with a state $\phi$, and $t > 0$. Through a Fock space construction, we define two states $\Phi_t$ and $\Psi_t$ on the tensor algebra $\mc{T}(\mc{B}, \phi)$ such that under the natural map $(\mc{B}, \phi) \rightarrow (\mc{T}(\mc{B}, \phi), \Phi_t, \Psi_t)$, free independence of arguments leads to free independence, while Boolean independence of centered arguments leads to conditionally free independence. The construction gives a new operator realization of the $(1+t)$'th free convolution power of any joint (star) distribution. We also compute several von Neumann algebras which arise.
\end{abstract}

%\begin{center}
%\textsl{Preliminary version}
%\end{center}

\maketitle

%\tableofcontents

%\bigskip

\section{Introduction}

Early results on the addition of freely independent random variables and free convolution emphasized the parallels with the theorems in classical probability. One of the first results exhibiting qualitatively new behavior in the free case was obtained by Nica and Speicher in
\cite{Nica-Speicher-Multiplication}. They showed that for any (joint) distribution $\mu$, its free convolution power $\mu^{\boxplus (1+t)}$ is defined for any real $t \geq 0$. Related results were obtained by random matrix methods by Voiculescu in the appendix of \cite{Nica-Speicher-Multiplication}; by a Fock space representation \cite{Shl97b}; by complex analytic methods in \cite{BerVoiSuperconvergence,Belinschi-Bercovici-Partially-defined}; and in the operator-valued setting in \cite{Ans-Bel-Fev-Nica,Shl-Free-convolution}.

Classically, one can only expect such behavior for infinitely divisible distributions. So the key illustrative example was the Bernoulli distribution, whose convolution powers (the binomial distributions) are only defined for integer powers, but whose free convolution powers (the free binomial distributions) are defined for any power $1 + t \geq 1$. The free binomial distributions are the Kesten-McKay distributions in the symmetric case, and more generally belong to the free Meixner class.

In \cite{AnsFree-Meixner}, the first author introduced the class of multivariate free Meixner states as joint distributions of operators obtained via a Fock space construction with depth two action \cite{Ans-Mashburn-Nearest}. Again, free multinomial distribution was proved to be the key example of a non-infinitely-divisible multivariate free Meixner distribution. Here the free multinomial distribution is the free convolution power of the joint distribution of several orthogonal projections.

In this paper we show that the preceding paragraph is somewhat misleading. There is nothing special about the free multinomial distribution. For any fixed $t \geq 0$, and any $\ast$-algebra $\mc{B}$ with a state $\phi$, we construct a Fock space such that for any $n$-tuple of centered elements $f_1, \ldots, f_n \in \mc{B}$, the joint distribution of the corresponding operators $X(f_1), \ldots, X(f_n)$ on the Fock space is the $(1+t)$'th free convolution power of the joint distribution of $f_1, \ldots, f_n$. One gets the free multinomial distribution by taking the $f_i$'s to be centered versions of orthogonal projections. One difference from the construction in \cite{Shl97b} is that $X(f)$ is symmetric for $f$ symmetric, and thus the result also holds for their joint $\ast$-distributions.

We state the full results in terms of the tensor algebra. Denote $\mc{B}^\circ = \ker \phi$, and let $\mc{T}(\mc{B}, \phi) = \mc{T}(\mc{B}^\circ)$ be the tensor algebra of $\mc{B}^\circ$. It will be convenient to denote the element of $\mc{T}(\mc{B}, \phi)$ corresponding to $f \in \mc{B}^\circ$ by $X(f)$, so that $\mc{T}(\mc{B}, \phi)$ is identified with the algebra of polynomials in $\set{X(f) : f \in \mc{B}^\circ}$ (subject to linearity relations). We construct explicit Fock representations for two (families of) states $\Phi_t$ and $\Psi_t$ on $\mc{T}(\mc{B}, \phi)$ with the following properties.

\begin{Thm*}
\

\begin{itemize}
\item
For $\set{f_1, \ldots, f_n} \subset \mc{B}^\circ$, the joint $\ast$-distribution of $X(f_1), \ldots, X(f_n)$ in $(\mc{T}(\mc{B}, \phi), \Phi_t)$ is the $(1+t)$'th free convolution power of the joint $\ast$-distribution of $f_1, \ldots, f_n$ in $(\mc{B}, \phi)$.
\item
For $\set{f_1, \ldots, f_n} \subset \mc{B}^\circ$, the joint $\ast$-distribution of $X(f_1), \ldots, X(f_n)$ in $(\mc{T}(\mc{B}, \phi), \Psi_t)$ is a Boolean convolution power of their joint $\ast$-distribution in $(\mc{T}(\mc{B}, \phi), \Phi_t)$.
\item
If $\set{f_1, \ldots, f_n} \subset \mc{B}$ are $\ast$-freely independent in $(\mc{B}, \phi)$, then the corresponding operators $\set{X(f_1), \ldots, X(f_n)}$ are $\ast$-freely independent with respect to $\Phi_t$.
\item
If instead $\set{f_1, \ldots, f_n} \subset \mc{B}^\circ$ are $\ast$-Boolean independent in $(\mc{B}, \phi)$, then the corresponding operators $\set{X(f_1), \ldots, X(f_n)}$ are $\ast$-freely independent with respect to $\Psi_t$, and $\ast$-conditionally free with respect to the pair $(\Phi_t, \Psi_t)$.
\item
$\Phi_t$ is tracial if $\phi$ is.
\item
For $t > 0$, $\Phi_t$ is faithful on $\mc{T}(\mc{B}, \phi)$ if $\phi$ is.
\end{itemize}
\end{Thm*}

It is a familiar fact that under the Gaussian functor, orthogonal vectors in a Hilbert space correspond to independent (Gaussian) operators on the symmetric Fock space, and under the semicircular functor, orthogonal vectors in a Hilbert space correspond to free (semicircular) operators on the full Fock space. Similarly, the usual/free (compound) Poisson Segal quantization maps elements of an algebra which are strongly orthogonal (in the sense that their product is zero), to independent/free compound Poisson operators. The theorem above lists somewhat similar properties, with orthogonality of initial elements replaced by their free/Boolean independence.

As a by-product, we find a curious formula for Boolean cumulants of centered elements:
\[
B^\phi[f_1, \ldots, f_n] = \phi[f_1 \Lambda(f_2, \Lambda(f_3, \ldots, \Lambda(f_{n-1}, f_n)))]
\]
for $n \geq 3$, where $\Lambda(f,g) = f g - \phi[f g]$.

The second part of the article is concerned with the study of joint distributions of $d$-tuples of operators $\set{X(p_i^\circ): 1 \leq i \leq d}$, where each $p_i^\circ$ is a centered projection, and the von Neumann subalgebras they generate inside the von Neumann algebras $W^\ast(\mc{T}(\mc{B}, \phi), \Phi_t)$ arising in the GNS representation of $\mc{T}(\mc{B}, \phi)$ with respect to $\Phi_t$, for several choices of $(\mc{B}, \phi)$. We start with a single projection with $\phi[p] = \alpha$. See Section~\ref{Sec:Free-product} for some of the notation.

\begin{Thm}
\label{Thm:C2-Intro}
Let $\mc{B} = \underset{\alpha}{\mf{C}} \oplus \underset{1 - \alpha}{\mf{C}}$, with the state determined by the parameter $\alpha \in (0, \frac{1}{2})$. Then
\[
W^\ast\left(\mc{T}\left(\underset{\alpha}{\mf{C}} \oplus \underset{1 - \alpha}{\mf{C}}, \phi \right), \Phi_t \right) \simeq
\begin{cases}
\underset{t}{L^\infty[0,1]} \oplus \underset{1 - \alpha (1 + t)}{\mf{C}} \oplus \underset{\alpha (1 + t) - t}{\mf{C}}, & t < \frac{\alpha}{1 - \alpha} \\
\underset{\alpha + \alpha t}{L^\infty[0,1]} \oplus \underset{1 - \alpha (1 + t)}{\mf{C}}, & \frac{\alpha}{1 - \alpha} \leq t < \frac{1 - \alpha}{\alpha} \\
L^\infty[0,1], & \frac{1 - \alpha}{\alpha} \leq t
\end{cases}
\]
In this case the states $\Phi_t$ and $\Psi_t$ can be identified with measures, and are described explicitly in Theorem~\ref{Thm:C2}.
\end{Thm}

The next theorem considers the case when the projections are free.

\begin{Thm*}
Let $\mc{B}$ be the free product $\mc{B} = \ast_{i=1}^d \left( \underset{\alpha_i}{\overset{p_i}{\mf{C}}} \oplus \overset{1 - p_i}{\underset{1 - \alpha_i}{\mf{C}}} \right)$, the $i$'th copy generated by the identity and the projection $p_i$ of state $\alpha_i \leq \frac{1}{2}$. In $W^\ast(\mc{T}(\mc{B}, \phi), \Phi_t)$, the von Neumann subalgebra generated by $X(p_i^\circ)$ is
\[
W^\ast(X(p_i^\circ) : 1 \leq i \leq d) \simeq
\underset{1 - \gamma_1 - \gamma_2}{L(\mf{F}_x)} \oplus \underset{\gamma_1}{\mf{C}} \oplus \underset{\gamma_2}{\mf{C}},
\]
where $L(\mf{F}_x)$ is an interpolated free group factor of appropriate dimension, and
	\begin{align*}
		\gamma_1 &= \max\left\{ 1 - \left( \sum_{i=1}^{d}\alpha_1 \right)(1+t), 0\right\} \\
		\gamma_2 &= \max\left\{ \left( \alpha_d - \sum_{i=1}^{d-1}\alpha_i \right)(1+t) -t, 0 \right\}
	\end{align*}
A more detailed statement is given in Theorem~\ref{Thm:Free-product}.
\end{Thm*}

The next theorem considers the case when the projections are Boolean independent. We first observe that the Boolean product of $d$ copies of $\mf{C}^2$, with non-degenerate states, is $M_{d+1}(\mf{C})$ with a vector state.

\begin{Thm}
\label{Thm:Boolean-Phi}
For $1 \leq i \leq d$, let $p_i^\circ = \frac{1}{2} (E_{1, 1+i} + E_{1 + i, 1})$. In $W^\ast(\mc{T}(M_{d+1}(\mf{C}), \phi_{11}), \Phi_t)$, the von Neumann subalgebra generated by $X(p_i^\circ)$ is
\begin{equation*}
W^\ast(X(p_i^\circ) : 1 \leq i \leq d)
\simeq \begin{cases}
L(\mf{F}_d) \oplus \mf{B}(\ell_2) & \text{ if } t < \sqrt{d}, \\
L(\mf{F}_d) & \text{ if } t\geq \sqrt{d}.
\end{cases}
\end{equation*}
Also, in $W^\ast(\mc{T}(\mf{B}(\mc{H}), \phi), \Phi_t)$ for a vector state $\phi$, $W^\ast(X(p_i^\circ) : i \in \mf{N}) \simeq \mf{B}(\ell^2)$.
\end{Thm}

Finally, in the case $\mc{B} = L^\infty[0,1]$, we give a complete description of the von Neumann algebra.

\begin{Thm}
\label{Thm:L-infty}
For $t > 0$, $W^\ast(\mc{T}(L^\infty[0,1], dx), \Phi_t) \simeq L(\mf{F}_{1 + 2 t})$.
\end{Thm}

The paper is organized as follows. In Section~\ref{Section:Construction}, we introduce three Fock space constructions, all of which are particular cases of the general construction in \cite{Ans-Mashburn-Nearest}. We show that in the first construction, for $t=0$ the algebra of operators degenerates to the original probability space $(\mc{B}, \phi)$. Then we pull back the vacuum states of two different constructions to the same tensor algebra. Next, we prove the aforementioned results about different types of cumulants and independence. In Section~\ref{Sec:vN}, we compute the von Neumann algebras.

This article, along with \cite{Ans-Mashburn-Nearest}, forms a part of the second author's Ph.D.~ thesis.

\textbf{Acknowledgements.} The first author is grateful to Zhiyuan Yang for comments leading to Proposition~\ref{Prop:Compression}. The authors are also grateful to the referee for useful comments.

\section{Background}

\subsection{Partitions}

 A \textit{partition} $\pi$ of a subset $S\subset \mathbb{N}$ is a collection of disjoint subsets of $S$ (called $blocks$ of $\pi$) whose union equals $S$. We will use $i \overset{\pi}{\thicksim} j$ to say that $i$ and $j$ are in the same block of $\pi$. In this paper, we will only be concerned with partitions of $[n] :=\{1, 2, \ldots, n\}$.

Let $\mathrm{NC}(n)$ denote the set of \textit{noncrossing partitions} over $[n]$, that is, those partitions $\pi$ such that there are no $i < j < k < \ell$ such that $i \overset{\pi}{\thicksim} k$ and $j \overset{\pi}{\thicksim} \ell$ unless all four are in the same block.
%If $i\in[n]$ is the first element of its block, we will call it an \textit{opening element}, while the last of its block will be called a \textit{closing element}.

A block $V$ of a noncrossing partition $\pi$ is called inner if for some other block $W \in \pi$, there exist $j \in V$ and $i, k \in W$ such that $i < j < k$. Otherwise $V$ is called outer. Denote $\Inner(\pi)$ the inner blocks of $\pi$, and $\Outer(\pi)$ the outer blocks.

%For distinct blocks $V$ and $W$ of a noncrossing partition $\pi$, $V$ is said to be \textit{inner with respect to} $W$ if $o_W < o_V < c_V < c_W$, where $o_V$ and $o_W$ are the opening elements of $V$ and $W$, respectively, while $c_V$ and $c_W$ are their closing elements. We will simply say a block is \textit{inner} if it is inner with respect to some block, and \textit{outer} if it is not.

The set $\widetilde{\mathrm{NC}}(n)$ of irreducible partitions is the set of noncrossing partitions with a single outer block, which is denoted by $V_{out}(\pi)$.  $\widetilde{\mathrm{NC}}_{ns}(n)$ is the set of all such partitions with no singleton blocks.

Finally, let $\mathrm{Int}(n)$ denote the \textit{interval partitions} over $[n]$, that is, those partitions $\pi$ such that whenever $i<j$ and $i \overset{\pi}{\thicksim} j$, we have $i \overset{\pi}{\thicksim} k$ for all $i < k < j$.

\subsection{Distributions and von Neumann algebras}

A noncommutative probability space $(\mc{A}, \phi)$ is a unital algebra $\mc{A}$ with a linear functional $\phi$ on it. A noncommutative $\ast$-probability space $(\mc{A}, \phi)$ is a unital $\ast$-algebra $\mc{A}$ with a positive linear functional $\phi$ on it.

For a noncommutative $\ast$-probability space $(\mc{A}, \phi)$, denote by $L^2(\mc{A}, \phi)$ the corresponding GNS Hilbert space. If $\mc{A}$ is represented on $L^2(\mc{A}, \phi)$ by bounded operators (in particular, if $\mc{A}$ is a $C^\ast$-algebra), denote by $W^\ast(\mc{A}, \phi)$ the von Neumann algebra obtained as the weak closure of $\mc{A}$ in this representation.

We will also need the following result, which follows easily from the uniqueness of the GNS construction.

\begin{Lemma}
\label{Lemma:vN-Rep}
Let $(\mc{A}, \phi)$ and $(\mc{B}, \psi)$ be two noncommutative $\ast$-probability spaces, and $F : \mc{A} \rightarrow \mc{B}$ a surjective $\ast$-homomorphism such that $\phi = \psi \circ F$. Then $W^\ast(\mc{A}, \phi) \simeq W^\ast(\mc{B}, \psi)$.
\end{Lemma}

\begin{comment}
The following is also well known.

\begin{Lemma}
\label{Lemma:Tensor-rep}
Suppose $\pi$ is a faithful representation of a $\ast$-algebra $\mc{A}$ on a Hilbert space $\mc{H}$. Then $\pi^{\otimes n}$ is a faithful representation of the algebraic tensor product $\mc{A}^{\otimes n}$ on $\mc{H}^{\otimes n}$.
\end{Lemma}
\end{comment}

\subsection{Independence}

Let $(\mc{A}, \phi)$ be a noncommutative probability space.
\begin{itemize}
\item
Subalgebras $\mc{A}_1, \ldots, \mc{A}_k \subset (\mc{A},\phi)$ are \textit{freely independent}, or \textit{free}, with respect to $\phi$ if whenever
\begin{equation*}
a_i\in\mc{A}_{u(i)}\text{ such that }u(1)\neq u(2)\neq\ldots\neq u(n)\text{ and } \phi[a_i]=0 \,\forall i=1,\ldots,n,\text{ then}
\end{equation*}
\begin{equation*}
\phi[a_1a_2\ldots a_n] = 0.
\end{equation*}
Elements $a_1, \ldots, a_n \in \mc{A}$ are $\ast$-free if the $\ast$-subalgebras they generate are freely independent.
\item
Subalgebras $\mc{A}_1, \ldots, \mc{A}_k \subset (\mc{A},\phi)$ that do not contain the unit are \textit{Boolean independent} with respect to $\phi$ if whenever
\begin{equation*}
	a_i\in\mc{A}_{u(i)}\text{ such that }u(1)\neq u(2)\neq\ldots\neq u(n), \text{ then}
\end{equation*}
\begin{equation*}
	\phi[a_1a_2\ldots a_n] = \phi[a_1]\phi[a_2]\ldots\phi[a_n].
\end{equation*}
Elements $a_1, \ldots, a_n \in \mc{A}$ are $\ast$-Boolean independent if the non-unital $\ast$-subalgebras they generate are freely independent.
\item
Let $\psi$ be a second state on $\mc{A}$. Subalgebras $\mc{A}_1, \ldots, \mc{A}_k \subset \mc{A}$ are \textit{conditionally free} in $(\mc{A}, \phi, \psi)$ if whenever
\[
a_i\in\mc{A}_{u(i)} \text{ such that }u(1)\neq u(2)\neq \ldots\neq u(n) \text{ and } \psi[a_i] = 0 \,\forall i=1,\ldots,n,\text{ then}
\]
\[
\phi[a_1 a_2 \ldots a_n] = \phi[a_1]\phi[a_2]\ldots\phi[a_n].
\]
\end{itemize}

\subsection{Cumulants}

For a family of multilinear functionals $\set{F_n : n \in \mf{N}}$ and a partition $\pi \in \mathrm{NC}(n)$, denote
\[
F_\pi[a_1, \ldots, a_n] = \prod_{V \in \pi} F_{\abs{V}}[a_i : i \in V],
\]
where the argument of each factor are taken in the order of their appearance in the block.

Let $a_1, \ldots, a_n \in \mc{A}$.
\begin{itemize}
\item
Their free cumulants $R^\phi[a_1, \ldots, a_n]$ with respect to $\phi$ are defined recursively via
\begin{equation*}
\phi[a_1, \ldots, a_n] = \sum_{\pi \in \mathrm{NC}(n)} R^\phi_\pi[a_1, \ldots, a_n].
\end{equation*}
\item
Their Boolean cumulants $B^\phi[a_1, \ldots, a_n]$ with respect to $\phi$ are defined recursively via
\begin{equation*}
\phi[a_1, \ldots, a_n] = \sum_{\pi \in \mathrm{Int}(n)} B^\phi_\pi[a_1, \ldots, a_n].
\end{equation*}
\item
Their conditionally free cumulants $R^{\phi, \psi}[a_1, \ldots, a_n]$ with respect to $(\phi, \psi)$ are defined recursively via
\begin{equation*}
\phi[a_1, \ldots, a_n] = \sum_{\pi \in \mathrm{NC}(n)} \left(\prod_{V \in \Inner(\pi)} R^\psi_{\abs{V}} [a_i : i \in V] \right) \left(\prod_{U \in \Outer(\pi)} R^{\phi, \psi}_{\abs{U}} [a_i : i \in U] \right).
\end{equation*}
\end{itemize}

\begin{Thm*}
\

\begin{itemize}
\item
Subalgebras $\mc{A}_1, \ldots, \mc{A}_n\subset (\mc{A},\phi)$ are freely independent if and only if for any $a_i \in \mc{A}_{u(i)}$,
\[
R^\phi[a_1, \ldots, a_n] = 0
\]
unless all $u(1) = u(2) = \ldots$.
\item
Non-unital subalgebras $\mc{A}_1, \ldots, \mc{A}_n\subset (\mc{A},\phi)$ are Boolean independent if and only if for any $a_i \in \mc{A}_{u(i)}$,
\[
B^\phi[a_1, \ldots, a_n] = 0
\]
unless all $u(1) = u(2) = \ldots$.
\item
Suppose the subalgebras $\mc{A}_1, \ldots, \mc{A}_n\subset (\mc{A},\phi, \psi)$ are freely independent in $(\mc{A}, \psi)$. They are conditionally free if and only if for any $a_i \in \mc{A}_{u(i)}$,
\[
R^{\phi, \psi}[a_1, \ldots, a_n] = 0
\]
unless all $u(1) = u(2) = \ldots$.
\end{itemize}
\end{Thm*}

As usual, we denote by $\boxplus, \uplus$ the (additive) free and Boolean convolutions.

\subsection{Free products}
\label{Sec:Free-product}

We will use the notation from \cite{Dykema-Free-products-hyperfinite}. For von Neumann algebras $A$ and $B$ with traces $\tau_A$ and $\tau_B$, $\underset{\alpha}{A} \oplus \underset{\beta}{B}$, where $\alpha, \beta \geq 0$ and $\alpha + \beta = 1$, will denote the algebra $A \oplus B$ whose associated trace $\tau(a,b) = \alpha \tau_A(a) + \beta \tau_B(b)$. In the case that one of the numbers $\alpha, \beta$ equals zero, say for example $\beta = 0$, we will take $\underset{\alpha}{A} \oplus \underset{\beta}{B}$ to denote the algebra $A$ with trace $\tau_A$. If we also want to give names to the projections corresponding to the identity elements of $A$ and $B$, we will write $\underset{\alpha}{\overset{p}{A}} \oplus \overset{q}{\underset{\beta}{B}}$, meaning $p = (1,0)$ and $q = (0,1)$. Similar notation will apply to direct sums of more that two algebras.

For $r > 1$, Dykema and Radulescu have independently introduced interpolated free group factors $L(\mf{F}_r)$. These are II$_1$ factors such that for $r \in \mf{N}$, $L(\mf{F}_r)$ is the von Neumann algebra of the free group on $r$ generators. We briefly recall the definition.

In a $W^\ast$-noncommutative probability space $(\mc{A}, \tau)$, let $\mc{R}$ be a copy of the hyperfinite II$_1$ factor, and $\omega = \set{X^{(t)} : t \in T}$ a free semicircular family free from $\mc{R}$. Choose self-adjoint projections $p_t \in \mc{R}$ such that $r = \sum_{t \in T} \tau[p_t]^2$. Then $L(\mf{F}_r)$ is isomorphic to $W^\ast(\mc{R}, \set{p_t X^{(t)} p_t : t \in T})$. Then the isomorphism class of $L(\mf{F}_r)$ depends only on $r$ and not on the choice of the projections. Dykema and Radulescu also proved that for any $r, s > 1$, the free product $L(\mf{F}_r) \ast L(\mf{F}_s) \simeq L(\mf{F}_{r+s})$, as well as the compression relation
\[
L(\mf{F}_r)_s = L\left(\mf{F}\left(1 + \frac{r-1}{s^2} \right)\right).
\]
In \cite{Dykema-Free-products-hyperfinite}, Dykema derived formulas for the reduced free products of direct sums of finite dimensional algebras and free group factors, with faithful tracial states. We refer the reader to that paper for specific results, and in particular for the notion of free dimension.

\subsection{Fock spaces with depth two action}

Finally, we recall the main construction from \cite{Ans-Mashburn-Nearest}. It works for a general map $\gamma : \mc{A} \rightarrow \mc{A}$ such that $\gamma + \rho = \gamma + \rho \mb{1}_{\mc{A}}$ is completely positive (semi-definite), but here we only state if for $\gamma$ scalar-valued.

\begin{Construction}
\label{Construction:gamma}
Let $\mc{A}$ be a unital $\ast$-algebra, equipped with star-linear maps $\rho, \gamma : \mc{A} \rightarrow \mf{C}$ and $\Lambda: \mc{A} \otimes_{\text{alg}} \mc{A} \rightarrow \mc{A}$ such that $\rho$ and $\gamma + \rho$ are positive, $\rho$ is faithful, and $\Lambda$ satisfies
\begin{equation}
\label{Eq:a0-symmetric}
\rho[g^\ast \Lambda(b \otimes f)] = \rho[\Lambda(b^\ast \otimes g)^\ast f], \quad \gamma[g^\ast \Lambda(b \otimes f)] = \gamma[\Lambda(b^\ast \otimes g)^\ast f],
\end{equation}
On the algebraic Fock space,
\[
\mc{F}_{alg}(\mc{A}) = \mf{C} \Omega \oplus \bigoplus_{n=1}^\infty \mc{A}^{\otimes n},
\]
define the inner product by the linear extension of
\begin{equation}
\label{Eq:Inner-product}
\ip{f_1 \otimes \ldots \otimes f_n}{g_1 \otimes \ldots \otimes g_k}_{\gamma,\rho} = \delta_{n=k} \rho[g_n^\ast f_n] \prod_{i=1}^{n-1} (\gamma + \rho)[g_{i}^\ast f_i].
\end{equation}
By Remark~3 from \cite{Ans-Mashburn-Nearest}, the inner product is positive definite provided that $\gamma + t \rho$ is still positive for some $t < 1$.
Next, for each $b \in \mc{A}$, consider operators on $\mc{F}_{alg}(\mc{A})$
\[
a^+(b) (f_1 \otimes \ldots \otimes f_n) = b \otimes f_1 \otimes \ldots \otimes f_n,
\]
\[
a^-(b) (f_1 \otimes \ldots \otimes f_n) = (\gamma + \rho)[b f_1] f_2 \otimes \ldots \otimes f_n,
\]
\[
a^-(b) (f_1) = \rho[b f_1] \Omega,
\]
\[
a^0(b) (f_1 \otimes \ldots \otimes f_n) = \Lambda(b \otimes f_1) \otimes  f_2 \otimes \ldots \otimes f_n,
\]
\[
a^-(b) (\Omega) = a^0(b)(\Omega) = 0,
\]
and
\[
X(b) = a^+(b) + a^-(b) + a^0(b).
\]
By Remark~3 from \cite{Ans-Mashburn-Nearest}, $X(b^\ast) = X(b)^\ast$.
\end{Construction}

\begin{Prop}
\label{Prop:AM22}
The properties below follow from the results \cite{Ans-Mashburn-Nearest} by restricting $\gamma : \mc{A} \rightarrow \mc{A}$ to be scalar-valued.
\begin{enumerate}
\item
$\Omega$ is cyclic for the algebra $\mathrm{Alg}_{\mf{C}}(\set{X(b) : b \in \mc{A}})$.
\item
For $V = \set{i(1) < \ldots < i(k)}$, denote
\begin{equation}
\label{Eq:WV}
W(V) = f_{i(1)} \Lambda(f_{i(2)}, \Lambda(f_{i(3)}, \ldots, \Lambda(f_{i(k-1)}, f_{i(k)}))).
\end{equation}
Then for $n\geq 2$ and $f_1, \ldots, f_n\in\mathcal{B}$, with respect to the vacuum state $\ip{\cdot \Omega}{\Omega}$, the moments (Proposition 5) are \begin{equation*}
M_n[X(f_1), \ldots, X(f_n)] = \sum_{\pi\in\mathrm{NC}_{ns}(n)} \prod_{V \in \Inner(\pi)} \rho[W(V)] \prod_{U \in \Outer(\pi)} (\gamma + \rho)[W(U)],
\end{equation*}
the Boolean cumulants (Lemma 7) are
\begin{equation*}
B_n[X(f_1), \ldots, X(f_n)] = \sum_{\pi\in\widetilde{\mathrm{NC}}_{ns}(n)} \rho[W(V_{out}(\pi))] \prod_{V \in \Inner(\pi)} (\gamma + \rho)[W(V)],
\end{equation*}
and the free cumulants (Proposition 9) are
\begin{equation*}
R_n[X(f_1), \ldots, X(f_n)] = \sum_{\pi\in\widetilde{\mathrm{NC}}_{ns}(n)} \rho[W(V_{out}(\pi))] \prod_{V \in \Inner(\pi)} \gamma[W(V)].
\end{equation*}
\item
Proposition 28: If $\mc{A}$ is a $C^\ast$-algebra, for $b \in \mc{A}$,
\[
\norm{a^+(b)}_{\gamma, \rho} = \norm{a^-(b^\ast)}_{\gamma, \rho} \leq \sqrt{\max(\norm{\rho}, \norm{\gamma + \rho})} \ \norm{b}.
\]
\item
Theorem 36: Suppose that $\rho$ is faithful, the inner product \eqref{Eq:Inner-product} is non-generate, and the operators $\set{X(b) : b \in \mc{A}}$ are bounded. Denote
\[
\Gamma_{\gamma, \Lambda}(\mc{A}, \rho) = W^\ast(X(f) : f \in \mc{A}) = W^\ast(X(f) : f \in \mc{A}^{sa}).
\]
Suppose that the following conditions hold:
\begin{equation*}
\label{Eq:Trace-star}
\Lambda(g^\ast \otimes f^\ast)^\ast = \Lambda(f \otimes g),
\end{equation*}
\begin{equation*}
\label{Eq:Trace-Associative}
f \gamma[g h] - \gamma[f g] h = \Lambda(f \otimes \Lambda(g \otimes h)) - \Lambda(\Lambda(f \otimes g) \otimes h),
\end{equation*}
and
\begin{equation*}
\label{Eq:Trace:gamma-self-adjoint}
\gamma[f] \rho[g] = \rho[f] \gamma[g].
\end{equation*}
Then the vacuum vector is cyclic and separating for $\Gamma_{\gamma, \Lambda}(\mc{A}, \rho)$. If in addition $\rho$ is tracial, the vacuum state is tracial on $\Gamma_{\gamma, \Lambda}(\mc{A}, \rho)$.
\end{enumerate}
\end{Prop}

In this paper we will consider the case where $\gamma$ is scalar-valued, and moreover $\gamma$ and $\rho$ are proportional to each other. A priori this gives a two-parameter family, but we show that a particular choice of a one-parameter family is natural.

\begin{Remark}
In the case of $\gamma : \mc{A} \rightarrow \mc{A}$, it is essential that $\mc{A}$ be an algebra. In the scalar-valued construction above, it suffices for $\mc{A}$ to be a vector space such that $\Lambda : \mc{A} \otimes \mc{A} \rightarrow \mc{A}$.

In \cite{Ans-Mashburn-Nearest}, it was also assumed that $\rho$ is faithful. The conclusions in Proposition~\ref{Prop:AM22}(a,b,c) in fact hold without this assumption.
\end{Remark}

\section{Constructions and cumulants}
\label{Section:Construction}

\subsection{States on the tensor algebra}
\label{Subsec:Tensor}

\begin{Defn}
Let $\mc{V}$ be a vector space with an involution $\ast$. Denote by $\mc{T}(\mc{V}) = \mf{C} 1 \oplus \bigoplus_{n=1}^\infty \mc{V}^{\otimes n}$ the tensor algebra of $\mc{V}$, with tensor multiplication and involution. It has the universal property that any $\ast$-linear map from $\mc{V}$ to a $\ast$-algebra $\mc{A}$ has a unique extension to a $\ast$-homomorphism from $\mc{T}(\mc{V})$ to $\mc{A}$.
\end{Defn}

Throughout the paper, let $\mc{B}$ be a unital $\ast$-algebra, and $\phi$ a unital $\ast$-linear functional on $\mc{B}$. Denote $\mc{B}^\circ = \ker \phi$. Then we can form the tensor algebra $\mc{T}(\mc{B}^\circ)$, which (to emphasize the dependence on $\phi$) we will denote typically denote $\mc{T}(\mc{B}, \phi)$.

\begin{Prop}
$\mc{T}(\mc{B}, \phi)$ does not depend on $\phi$.
\end{Prop}

\begin{proof}
Take the quotient of the tensor algebra $\mc{T}(\mc{B})$ by the relation
\begin{equation}
\label{Eq:relation}
f_1 \otimes \ldots \otimes f_{i-1} \otimes 1_{\mc{B}} \otimes f_{i+1} \otimes \ldots \otimes f_n = f_1 \otimes \ldots \otimes f_{i-1} \otimes f_{i+1} \otimes \ldots \otimes f_n,
\end{equation}
for each $n$ and each $1 \leq i \leq n$. We denote the resulting algebra by $\mc{T}_{r}(\mc{B})$ and call it the reduced tensor algebra of $\mc{B}$. We now claim that for any $\phi$, $\mc{T}(\mc{B}^\circ) = \mc{T}(\mc{B}, \phi)$ is $\ast$-isomorphic to $\mc{T}_{r}(\mc{B})$. Indeed, the linear extension of the map $\mc{T}_{r}(\mc{B}) \rightarrow \mc{T}(\mc{B}^\circ)$,
\[
f_1 \otimes \ldots \otimes f_n \mapsto \sum_{S \subset [n]} \prod_{i \in S} \phi[f_i] \bigotimes_{j \not \in S} f_i^\circ
\]
is well defined on $\mc{T}_{r}(\mc{B})$, and is a bijective $\ast$-homomorphism.
\end{proof}

To distinguish elements $f \in \mc{B}$ from their counterparts in $\mc{T}_r(\mc{B})$, we will denote the element $f_1 \otimes \ldots \otimes f_n \in \mc{T}_r(\mc{B})$ by $X(f_1) \ldots X(f_n)$, with $X(1_{\mc{B}}) = 1$. This notation is consistent with the multiplication, involution, and relation on $\mc{T}_r(\mc{B})$, as well as the isomorphism with any $\mc{T}(\mc{B}^\circ) = \mc{T}(\mc{B}, \phi)$.

\begin{Defn}
\label{Defn:Phi_t}
For $t \geq 0$, define two families of states on $\mc{T}(\mc{B}^\circ) = \mc{T}(\mc{B}, \phi)$ by
\begin{equation}
\label{Eq:Phi_t}
\Phi_t[X(f_1) \ldots X(f_n)] = \sum_{\pi\in\mathrm{NC}_{ns}(n)} (1 + t)^{\abs{\Outer(\pi)}} t^{\abs{\Inner(\pi)}} \prod_{V \in \pi} \phi[W(V)],
\end{equation}
and
\begin{equation}
\label{Eq:Psi_t}
\Psi_t[X(f_1) \ldots X(f_n)] = \sum_{\pi\in\mathrm{NC}_{ns}(n)} t^{\abs{\pi}} \prod_{V \in \pi} \phi[W(V)].
\end{equation}
\end{Defn}

We will have much better descriptions of these states in terms of their free, Boolean, and conditionally free cumulants.

\subsection{Three Fock space constructions}

\begin{Construction}
\label{Constr:Centered}
Let $(\mc{B}, \phi)$ be a noncommutative $\ast$-probability space, so that $\mc{B}$ is a unital $\ast$-algebra, and $\phi$ is a (not necessarily faithful) state on it. Using the GNS construction for $\mc{B}$, we may construct a pointed Hilbert space $\mc{H}$ (that is, a Hilbert space with a unit vector $\Omega$) such that
\begin{itemize}
\item
$\mc{B}$ acts on  $\mc{H}$ by densely defined, possibly unbounded operators whose domains contain $\Omega$.
\item
$\Omega$ is cyclic for $\mc{B}$, and $\phi = \ip{\cdot \Omega}{\Omega}$.
\end{itemize}
%\cred{Throughout the paper, we will assume that the representation of $\mc{B}$ on $\mc{H}$ is faithful.}
Denote by $\mc{H}^\circ$ the orthogonal complement in $\mc{H}$ of $\mf{C} \Omega$. Note that for any $f \in \mc{B}$, $f \Omega - \phi[f] \Omega \in \mc{H}^\circ$. Denote
\[
\mc{B}^\circ = \set{g \in \mc{B} : \phi[g] = 0},
\]
so that $\mc{B} = \mc{B}^\circ \oplus \mf{C}$. Finally denote
\begin{equation}
\label{Eq:Lambda}
\Lambda(f, g) = f g - \phi[f g].
\end{equation}
Fix $t \geq 0$. On the Fock space $\mc{F}(\mc{H}^\circ, t) = \mf{C} \Omega \oplus \bigoplus_{n=1}^\infty (\mc{H}^\circ)^{\otimes n}$, define the inner product by the linear extension of
\[
\ip{\xi_1 \otimes \ldots \otimes \xi_n}{\eta_1 \otimes \ldots \otimes \eta_n}_t = \delta_{n=k} (t+1) t^{n-1} \prod_{i=1}^n \ip{\xi_i}{\eta_i}
\]
For $f \in \mc{B}^\circ$, define
\begin{align*}
a^+(f)(\xi_1 \otimes \ldots \otimes \xi_n) & = (f \Omega) \otimes \xi_1 \otimes \ldots \otimes \xi_n, \quad a^+(f)(\Omega) = f \Omega, \\
a^-_\phi(f)(\xi_1 \otimes \ldots \otimes \xi_n) & = \ip{f \xi_1}{\Omega} \xi_2 \otimes \ldots \otimes \xi_n, \\
a^-(f) & = \begin{cases}
t a^-_\phi(f) & \text{ on } (\mc{H}^\circ)^{\otimes n}, n \geq 2, \\
(1 + t) a^-_\phi(f) & \text{ on } (\mc{H}^\circ)^{\otimes n}, n = 1, \\
0 & \text{ on } (\mc{H}^\circ)^{\otimes n}, n = 0,
\end{cases} \\
a^0(f) (\xi_1 \otimes \ldots \otimes \xi_n) & = (f \xi_1 - \ip{f \xi_1}{\Omega} \Omega) \otimes \xi_2 \otimes \ldots \otimes \xi_n, \quad a^0(f) \Omega = 0.
\end{align*}
For $f \in \mc{B}^\circ$, denote
\[
X(f, t) = a^+(f) + a^-(f) + a^0(f).
\]
For $f \in \mc{B}$, denote $X(f,t) = X(f - \phi[f], t) + \phi[f]$.
\end{Construction}

The specific choice of the annihilation operator (and inner product) is justified in Remark~\ref{Remark:Key-properties} and \ref{Prop:s-neq-t}.

\begin{Notation}
Denote
\[
\Gamma^\Phi_{a}(\mc{B}, \phi; t) = \mathrm{Alg}_{\mf{C}}(\set{X(f,t) : f \in \mc{B}^\circ}, 1)
= \mathrm{Alg}_{\mf{C}}(\set{X(f,t) : f \in (\mc{B}^\circ)^{sa}}, 1).
\]
Denote by
\[
\Gamma^\Phi_{w}(\mc{B}, \phi; t) = W^\ast(\set{X(f,t) : f \in \mc{B}^\circ}, 1)
\]
the corresponding von Neumann algebra. The vacuum state $\Phi = \ip{\cdot \Omega}{\Omega}$ is a state on each of these algebras.
\end{Notation}

\begin{Remark}
\label{Remark:Key-properties}
Except for $\mc{B}^\circ$ not necessarily being an algebra, and the linear functional not necessarily being faithful, the construction above fits into the framework of \cite{Ans-Mashburn-Nearest}, with the linear functional $\rho = (1 + t) \phi$, the map $\gamma = -\phi$, and $\Lambda(f \otimes g) = f g - \phi[f g]$. Therefore the following properties are either clear, or can be read off Proposition~\ref{Prop:AM22}.

\begin{enumerate}
\item
The inner product $\ip{\cdot}{\cdot}_t$ is non-degenerate for $t > 0$.
\item
Each $X(f^\ast, t) = X(f, t)^\ast$ in the natural sense.
\item
$\Omega$ is cyclic for $\Gamma^\Phi_{a}(\mc{B}, \phi; t)$.
\item
For $n \geq 2$ and $f_1, \ldots, f_n\in\mathcal{B}^\circ$, the moments with respect to $\Phi$ are
\begin{equation}
\label{Eq:Moments}
\Phi[X(f_1, t), \ldots, X(f_n, t)] = \sum_{\pi\in\mathrm{NC}_{ns}(n)} (1 + t)^{\abs{\Outer(\pi)}} t^{\abs{\Inner(\pi)}} \prod_{V \in \pi} \phi[W(V)],
\end{equation}
Boolean cumulants with respect to $\Phi$ are
\begin{equation}
\label{Eq:Boolean-cumulants}
B^\Phi[X(f_1, t), \ldots, X(f_n, t)] =  (1 + t) \sum_{\pi\in\widetilde{\mathrm{NC}}_{ns}(n)} t^{\abs{\pi} - 1} \prod_{V \in \pi} \phi[W(V)],
\end{equation}
and the free cumulants with respect to $\Phi$ are
	\begin{equation}
\label{Eq:Free-cumulants}
	R^\Phi[X(f_1, t), \ldots, X(f_n, t)] = (1 + t) \sum_{\pi\in\widetilde{\mathrm{NC}}_{ns}(n)} (-1)^{\abs{\pi} - 1} \prod_{V \in \pi} \phi[W(V)].
	\end{equation}
\item
If $\mc{B}$ is a $C^\ast$-algebra, then, denoting by $\norm{\cdot}_t$ the operator norm on $\mc{F}(\mc{H}^\circ, t)$, we have $\norm{X(f,t)}_t \leq 2(1 + \sqrt{1 + t}) \norm{f}_\mc{B}$. Indeed, Proposition~\ref{Prop:AM22}(c) gives
\begin{equation*}
\norm{a^+(f)}_t = \norm{a^-(f)}_t \leq \sqrt{1 + t} \norm{f}_\mc{B},
\end{equation*}
and it is not hard to check that $\norm{a^0(f)}_t \leq \norm{f}_\mc{B}$.
\item
If $t > 0$, $\phi$ is faithful, and $\mc{B}$ is a $C^\ast$-algebra, the assumptions in Proposition~\ref{Prop:AM22}(d) hold. The key calculation is that for $f, g, h \in \mc{B}^\circ$,
\[
\Lambda(f \otimes \Lambda(g \otimes h)) - \Lambda(\Lambda(f \otimes g) \otimes h)
= - f \phi[g h] + \phi[f g] h.
\]
Therefore in this case, $\Omega$ is cyclic and separating for $\Gamma^\Phi_{w}(\mc{B}, \phi; t)$. If $\phi$ is tracial, then $\Phi$ is tracial.
\end{enumerate}
\end{Remark}

Next we consider the case $t=0$. Compare with Proposition~4.8 in \cite{Belinschi-Nica-Eta} or Proposition~A.9 in \cite{AnsBoolean}.

\begin{Prop}
\label{Prop:t=0}
The noncommutative $\ast$-probability spaces $(\Gamma^\Phi_a(\mc{B}, \phi; 0), \Phi)$ and $(\mc{B}, \phi)$ are isomorphic. If $\mc{B}$ is a von Neumann algebra, then $\Gamma^\Phi_a(\mc{B}, \phi; 0) = \Gamma^\Phi_w(\mc{B}, \phi; 0)$.
\end{Prop}

\begin{proof}
If $t = 0$, the Fock space is simply $\mf{C} \Omega \oplus \mc{H}^\circ \simeq \mc{H}$. A short calculation shows that for $f \in \mc{B}$ and $\xi \in \mc{H}$, $X(f, 0) \xi = f \xi$, and $\ip{X(f, 0) \Omega}{\Omega} = \phi[f]$. The remaining claims follow.
\end{proof}

\begin{Cor}
\label{Cor:Boolean-cumulants}
Let $(\mc{B}, \phi)$ be a unital noncommutative probability space. Denote $\Lambda(f, g) = f g - \phi[f g]$. Then the Boolean cumulants of $f_1, \ldots, f_n \in \mc{B}^\circ$ are $B^\phi[f_1, f_2] = \phi[f_1 f_2]$ and
\[
B^\phi[f_1, \ldots, f_n] = \phi[f_1 \Lambda(f_2, \Lambda(f_3, \ldots, \Lambda(f_{n-1}, f_n)))].
\]
\end{Cor}

\begin{proof}
Apply equation~\eqref{Eq:Boolean-cumulants} with $t=0$.
\end{proof}

\begin{Construction}
\label{Constr:Psi}
Let $\mc{B}, \phi, \mc{H}$ be as in Construction~\ref{Constr:Centered}. On $\mc{F}(\mc{H}^\circ, t)$, define the simpler inner product
\[
\ip{\xi_1 \otimes \ldots \otimes \xi_n}{\eta_1 \otimes \ldots \otimes \eta_n}_t = \delta_{n=k} t^{n} \prod_{i=1}^n \ip{\xi_i}{\eta_i}.
\]
For $f \in \mc{B}^\circ$, define $a^+(f)$ and $a^0(f)$ as in Construction~\ref{Constr:Centered}, and $a^-(f) = t a_\phi^-(f)$. Denote
\[
Y(f,t) = a^+(f) + a^-(f) + a^0(f)
\]
as before, and let $\Psi$ be the corresponding vacuum state.

This construction also fits into the framework of \cite{Ans-Mashburn-Nearest}, with the linear functional $\rho = t \phi$ and the map $\gamma = 0$. Therefore using Proposition~\ref{Prop:AM22} again,
\begin{itemize}
\item
The inner product $\ip{\cdot}{\cdot}_t$ is non-degenerate for $t > 0$.
\item
Each $Y(f^\ast, t) = Y(f, t)^\ast$ in the natural sense.
\item
$\Omega$ is cyclic for $\Gamma^\Psi_{a}(\mc{B}, \phi; t) = \mathrm{Alg}_{\mf{C}}(\set{Y(f,t) : f \in \mc{B}^\circ}, 1)$.
\item
For $n \geq 0$ and $f_1, \ldots, f_n \in \mc{B}^\circ$, the moments with respect to $\Psi$ are
\begin{equation}
\label{Eq:Moments-Psi}
\Psi[Y(f_1, t) \ldots Y(f_n, t)] = \sum_{\pi\in\mathrm{NC}_{ns}(n)} t^{\abs{\pi}} \prod_{V \in \pi} \phi[W(V)],
\end{equation}
and the free cumulants with respect to $\Psi$ are
\begin{equation}
\label{Eq:Free-cumulants-Psi}
R^\Psi[Y(f_1, t), \ldots, Y(f_n, t)] = t \phi[W(1, 2, \ldots, n)] = t \phi[f_1 \Lambda(f_2, \Lambda(f_3, \ldots, \Lambda(f_{n-1}, f_n)))].
\end{equation}
\item
If $\mc{B}$ is a $C^\ast$-algebra, then $\norm{Y(f,t)}_t \leq 2(1 + \sqrt{t}) \norm{f}_\mc{B}$.
\end{itemize}
\end{Construction}

\begin{Lemma}
\label{Lemma:Nontracial}
Suppose $\phi$ is faithful. Unless $\dim \mc{B}^\circ \leq 1$ or $\phi$ is a homomorphism, $\Psi$ is not tracial.
\end{Lemma}

\begin{proof}
Applying Proposition~\ref{Prop:AM22}(d) with $\gamma = 0$ and $\Lambda(f \otimes g) = f g - \phi[f g]$, the vacuum state is tracial only if $\phi$ is tracial and for any $f, g, h \in \mc{B}^\circ$, $\phi[f g] h - f \phi[g h] = 0$. Suppose $\dim \mc{B}^\circ \geq 2$. Let $f, h \in \mc{B}^\circ$ be linearly independent. Then for any $g \in \mc{B}^\circ$, $\phi[f g] = 0$, in other words $\mc{B}^\circ$ is a subalgebra. This implies that $\phi$ is a homomorphism.
\end{proof}

\begin{Thm}
\label{Thm:pi-faithful}

For $t \geq 0$, define the maps $\pi_{\Phi, t} : \mc{B} \rightarrow \Gamma^\Phi_a(\mc{B}, \phi; t)$ and $\pi_{\Psi, t} : \mc{B} \rightarrow \Gamma^\Psi_a(\mc{B}, \phi; t)$ by
\begin{align*}
\pi_{\Phi, t} : f & \mapsto X(f, t) \\
\pi_{\Psi, t}: f & \mapsto Y(f, t)
\end{align*}

\begin{enumerate}
\item
$\pi_{\Phi, t}$ and $\pi_{\Psi, t}$ extend to $\ast$-representations of the tensor algebra $\mc{T}(\mc{B}, \phi)$ on the Fock spaces from Constructions~\ref{Constr:Centered} and \ref{Constr:Psi}, respectively.
\item
$\Phi_t = \Phi \circ \pi_{\Phi, t}$ and $\Psi_t = \Psi \circ \pi_{\Psi, t}$. In particular, we have the equality of joint distributions
\[
\mu^\Phi_{X(f_1, t), \ldots X(f_n, t)} = \mu^{\Phi_t}_{X(f_1), \ldots X(f_n)}.
\]
\item
$W^\ast(\mc{T}(\mc{B}, \phi), \Phi_t) \simeq \Gamma^\Phi_w(\mc{B}, \phi; t)$ from Construction~\ref{Constr:Centered}, while $W^\ast(\mc{T}(\mc{B}, \phi), \Psi_t) \simeq \Gamma^\Psi_w(\mc{B}, \phi; t)$ from Construction~\ref{Constr:Psi}.
%\item
%For $t > 0$, if $f \in \mc{B}^\circ$ and $\pi_{\Phi, t}[f] \Omega = 0$, then $f = 0$.
\item
If $t > 0$ and $\phi$ is faithful, the representation $\pi_{\Phi, t}$ is faithful.
\item
If $t > 0$ and $\phi$ is faithful, the state $\Phi_t$ is faithful.
\end{enumerate}
\end{Thm}

\begin{proof}
Part (a) follows by the universal property of the tensor algebra. (b) follows by comparing Definition~\ref{Defn:Phi_t} with properties \eqref{Eq:Moments} and \eqref{Eq:Moments-Psi}. Since $X(f,t), Y(f,t)$ are bounded, and $\Omega$ is cyclic in both representations, (c) follows from the uniqueness of the GNS representation.

%For (d), since $\Omega$ is separating for $\Gamma^\Phi_a(\mc{B}, \phi, t)$, $\pi_{\Phi, t}[f] = X(f,t) = 0$. Thus in particular, for $\xi \in %\mc{H}^\circ$,
%\[
%X(f,t) \xi = f \Omega \otimes \xi + (f \xi - \ip{f \xi}{\Omega} \Omega) + (1+t) \ip{f \xi}{\Omega} \Omega = 0.
%\]
%It follows that for any $\xi \in \mc{H}$, $f \xi = 0$. Since the representation of $\mc{B}$ on $\mc{H}$ is assumed to be faithful, $f = 0$.

For (d), suppose $\pi_{\Phi, t}(\vec{\xi}) \Omega = 0$ for some $\vec{\xi} \in \mc{T}(\mc{B}, \phi)$. Denote the component of $\vec{\xi}$ in the highest tensor power by $\sum_{i=1}^{N} f_1^{(i)} \otimes \ldots \otimes f_{n}^{(i)}$. It suffices to show that this sum is zero.
\[
\begin{split}
0 & = \pi_{\Phi, t}[\vec{\xi}] \Omega = \sum_{i=1}^{N} X(f_1^{(i)}, t) \ldots  X(f_{n}^{(i)}, t) \Omega + \text{ terms in lower tensor components} \\
& = \sum_{i=1}^{N} f_1^{(i)} \Omega \otimes \ldots \otimes f_{n}^{(i)} \Omega + \text{ terms in lower tensor components},
 \end{split}
\]
and so
\[
\sum_{i=1}^{N} f_1^{(i)} \Omega \otimes \ldots \otimes f_{n}^{(i)} \Omega = 0
\]
as an element of $(\mc{H}^\circ)^{\otimes n}$. Let $\set{g_k \Omega: 1 \leq k \leq K}$ be a basis for the subspace
\[
\Span{\set{f_j^{(i)} \Omega : 1 \leq i \leq N, 1 \leq j \leq n}} \subseteq \mc{H}^\circ.
\]
In particular, for some coefficients, $f^{(i)}_j \Omega = \sum_{k=1}^K c^{(i)}_{j, k} g_k \Omega$. Since $\phi$ is faithful, $\Omega$ is separating for the representation of $\mc{B}$ on $\mc{H}$. Therefore $f^{(i)}_j = \sum_{k=1}^K c^{(i)}_{j, k} g_k$. So if
\[
0 = \sum_{i=1}^{N} f_1^{(i)} \Omega \otimes \ldots \otimes f_{n}^{(i)} \Omega = \sum_{i=1}^N \sum_{k : [n] \rightarrow [K]} \prod_{j=1}^n c^{(i)}_{j, k(j)} g_{k(1)} \Omega \otimes \ldots \otimes g_{k(n)} \Omega,
\]
then each $\sum_{i=1}^N \prod_{j=1}^n c^{(i)}_{j, k(j)} = 0$, and therefore $\sum_{i=1}^{N} f_1^{(i)} \otimes \ldots \otimes f_{n}^{(i)} = 0$.

Since $\phi$ is faithful, by Remark~\ref{Remark:Key-properties}(f), $\Omega$ is also separating for $\Gamma^\Phi_a(\mc{B}, \phi, t)$, and so $\Phi$ is faithful on it. Therefore (e) follows from (d).
\end{proof}

\begin{Construction}
\label{Constr:Ricard}
For comparison, we include a third Fock space construction considered by Bo\-{\.z}e\-jko, Wysocza{\'n}sky \cite{Boz-Wys} and Ricard \cite{Ricard-t-Gaussian}. It starts with a Hilbert space $\mc{H}$ and the inner product
\[
\ip{\xi_1 \otimes \ldots \otimes \xi_n}{\eta_1 \otimes \ldots \otimes \eta_n}_t = \delta_{n=k} t^{n-1} \prod_{i=1}^n \ip{\xi_i}{\eta_i}
\]
on the corresponding Fock space. For $f \in \mc{H}_{\mf{R}}$, define $a^+(f)$ as usual, and $a^-(f)$ as its adjoint, which comes out to be $t a^-_{free}(f)$ on $\mc{H}^{\otimes n}$, $n \geq 2$, and $a^-_{free}(f)$ on $\mc{H}$. Denote
\[
Z(f,t) = a^+(f) + a^-(f).
\]
Note the absence of the $a^0$ operator. Thus in the framework of \cite{Ans-Mashburn-Nearest}, we would take $\Lambda = 0$, the linear functional $\rho = \phi$, and the map $\gamma = (t-1) \phi$.
\end{Construction}

\subsection{Cumulants, convolutions, and independence}

\begin{Prop}
\label{Prop:Free-power}
Let $\set{f_1, \ldots, f_n} \subset \mc{B}^\circ$. Then
\[
R^{\Phi_t}[X(f_1), \ldots, X(f_n)] = (1 + t) R^\phi[f_1, \ldots, f_n].
\]
Therefore we have the relation between joint distributions
\[
\mu^{\Phi_t}_{X(f_1), \ldots X(f_n)} = \mu_{f_1, \ldots, f_n}^{\boxplus (1 + t)}.
\]
\end{Prop}

\begin{proof}
By Proposition \ref{Prop:t=0}, we have
\[
R^\phi[f_1, \ldots, f_n] = R^\Phi[X(f_1, 0), \ldots, X(f_n, 0)] = R^{\Phi_t}[X(f_1), \ldots, X(f_n)].
\]
From Equation~\eqref{Eq:Free-cumulants}, we see that
\[
R^\Phi[X(f_1, t), \ldots, X(f_n, t)] = (1 + t)R^\Phi[X(f_1, 0), \ldots, X(f_n, 0)].
\]
\end{proof}

\begin{Remark}
Let $(\mc{A}, \mf{E}, \mc{C})$ be a $\mc{C}$-valued probability space. For every c.p. map $\eta$ on $\mc{C}$, one can give a version of Construction~\ref{Constr:Centered}, with the inner product
\begin{multline*}
\ip{f_1 \otimes \ldots \otimes f_n}{g_1 \otimes \ldots \otimes g_n}_\eta \\
= \delta_{n=k} ((1 + \eta) \circ \mf{E}) \left[ g_n^\ast (\eta \circ \mf{E})  \left[ g_{n-1}^\ast (\eta \circ \mf{E}) \left[ \ldots (\eta \circ \mf{E}) [g_1^\ast f_1] \ldots \right] f_{n-1} \right] f_n \right]
\end{multline*}
on $\mc{C} \Omega \oplus \bigoplus_{n=1}^\infty (\mc{A}^\circ)^{\otimes n}$ (but not on $\mc{C} \Omega \oplus \bigoplus_{n=1}^\infty (\mc{A}^\circ)^{\otimes_\mc{C} n}$).
Then the corresponding forms of Proposition~\ref{Prop:t=0} and Corollary~\ref{Cor:Boolean-cumulants} hold with the same proof. In Proposition~\ref{Prop:Free-power}, the formula
\[
R[X(f_1, \eta), \ldots, X(f_n, \eta)] = (1 + \eta) [R[X(f_1, 0), \ldots, X(f_n, 0)]].
\]
holds as well. Note however that, according to the standard definition of $\mc{C}$-valued distribution, this is not sufficient to claim that the joint distribution of the variables on the left-hand side is the $(1 + \eta)$'th free convolution power of the joint distribution of the variables on the right-hand side.
\end{Remark}

\begin{Example}
If $p_1, \ldots, p_n$ are orthogonal projections adding up to the identity, then so are $X(p_1, 0)$, $\ldots$, $X(p_n, 0)$. Therefore the joint distribution of $X(p_1, t), \ldots, X(p_n, t)$ is free multinomial in the sense of Section 4.6 of \cite{AnsFree-Meixner}.
\end{Example}

\begin{Prop}\label{Prop:TwoStateFCs}
Let $\set{f_1, \ldots, f_n} \subset\mc{B}^\circ$.
\begin{enumerate}
\item
The free cumulants with respect to the state ${\Psi_t}$ are
\[
\begin{split}
R^{\Psi_t}[X(f_1), \ldots, X(f_n)]
& = t B^\phi[f_1, \ldots, f_n].
\end{split}
\]
Therefore we have the relation between joint distributions
\[
\mu^{\Psi_t}_{X(f_1), \ldots X(f_n)} = \mf{B}_t(\mu_{f_1, \ldots, f_n})^{\uplus t} = \left( \mu^{\Phi_t}_{X(f_1), \ldots X(f_n)} \right)^{\uplus \frac{t}{1+t}},
\]
where $\mf{B}_t$ is the Belinschi-Nica transformation.
\item
The conditionally free cumulants with respect to the pair $({\Phi_t}, {\Psi_t})$ are
\[
R^{({\Phi_t}, {\Psi_t})}[X(f_1), \ldots, X(f_n)]
= (1 + t) B^\phi[f_1, \ldots, f_n].
\]
\end{enumerate}
\end{Prop}

\begin{proof}
For (a), combine equation~\eqref{Eq:Free-cumulants-Psi} with Corollary~\ref{Cor:Boolean-cumulants}. To obtain the second identity, using  Proposition~3.5, equation (6.11), and Definition~4.1 from \cite{Belinschi-Nica-Free-BM},
\[
\begin{split}
B^{\Psi_t}[X(f_1), \ldots, X(f_n)]
& = \sum_{\pi \in \widetilde{NC}(n)} R^{\Psi_t}_\pi[X(f_1), \ldots, X(f_n)] \\
& = \sum_{\pi \in \widetilde{NC}(n)} t^{\abs{\pi}} B^\phi_\pi[f_1, \ldots, f_n] \\
& = t B_{\mf{B}_t(\mu_{f_1, \ldots, f_n})} \\
& = B_{\mf{B}_t(\mu_{f_1, \ldots, f_n})^{\uplus t}} \\
& = B_{\left( \mu_{f_1, \ldots, f_n}^{\boxplus (1 + t)} \right)^{\uplus \frac{t}{1+t}}}.
\end{split}
\]
Now combine with Proposition~\ref{Prop:Free-power}.

(b) First, $R^{({\Phi_t},{\Psi_t})}[X(f_1)] = 0$. Next,
\begin{align*}
\sum_{\pi \in \mathrm{Int}(n)} B^{\Phi_t} [X(f_1), \ldots , X(f_n)] &= {\Phi_t} [X(f_1) \ldots X(f_n)] \\
&= \sum_{\pi \in \mathrm{NC}(n)} \left( \prod_{V \in \Inner(\pi)} R^{\Psi_t} [f_i : i \in V] \right) \left( \prod_{U \in \Outer(\pi)} R^{{\Phi_t},{\Psi_t}} [f_i : i \in U] \right).
\end{align*}

The left-hand side is equal to
\begin{equation*}
\sum_{\pi \in \mathrm{Int}(n)} \prod_{V \in \pi} \left( (1+t) \sum_{\sigma\in\widetilde{\mathrm{NC}}_{ns}(V)} t^{|\sigma| - 1} \prod_{W\in\sigma} B^\phi [f_i : i \in W] \right),
\end{equation*}
while the right-hand side, by part (a), is equal to
\begin{equation*}
\sum_{\pi \in \mathrm{NC}(n)} \left( \prod_{V \in \Inner(\pi)} t B^\phi [f_i : i \in V] \right) \left( \prod_{U \in \Outer(\pi)} R^{{\Phi_t},{\Psi_t}} [f_i : i \in U] \right).
\end{equation*}
From this, the claim inductively follows, since each $\pi\in\mathrm{NC}(n)$ can be uniquely constructed from an interval partition by replacing each block $V$ with some $\sigma\in\widetilde{\mathrm{NC}}(V)$, a partition of the elements of the block.
\end{proof}

The following corollary follows directly from Propositions \ref{Prop:Free-power} and \ref{Prop:TwoStateFCs}.

\begin{Cor}
\label{Cor:Independence-functor}
\

\begin{enumerate}
\item
Suppose $\set{f_1, \ldots, f_n}$ are $\ast$-free in $(\mc{B}, \phi)$. Then $\set{X(f_1), \ldots, X(f_n)}$ are $\ast$-free in $(\mc{T}(\mc{B}, \phi), \Phi_t)$.
\item
Suppose $\set{f_1, \ldots, f_n} \subset \mc{B}^\circ$ are $\ast$-Boolean independent in $(\mc{B}, \phi)$. Then $\set{X(f_1), \ldots, X(f_n)}$ are $\ast$-free in $(\mc{T}(\mc{B}, \phi), \Psi_t)$ and $\ast$-conditionally free in $(\mc{T}(\mc{B}, \phi), \Phi_t, \Psi_t)$.
\end{enumerate}
\end{Cor}

\begin{Prop}
\label{Prop:Compression}
Let $q$ be a projection of trace $\frac{1}{1+t}$, so that we may identify $W^\ast(q)$ with $\overset{q}{\underset{\frac{1}{1+t}}{\mf{C}}} \oplus \overset{1-q}{\underset{\frac{t}{1+t}}{\mf{C}}}$ with the state $\tau$. Denote by $\mc{B} \ast W^\ast(q)$ the algebraic reduced free product, and consider the noncommutative probability space $\mc{A} = q (\mc{B} \ast W^\ast(q)) q$ with the state $\tilde{\phi} = (1 + t) (\phi \ast \tau)|_{\mc{A}}$. Then
\[
W^\ast (\mc{T}(\mc{B}, \phi), \Phi_t) \simeq W^\ast(\mc{A}, \tilde{\phi}).
\]
In particular, when $\mc{B} \ast \left( \overset{q}{\underset{\frac{1}{1+t}}{\mf{C}}} \oplus \overset{1-q}{\underset{\frac{t}{1+t}}{\mf{C}}} \right)$ is a II$_1$ factor,
\[
W^\ast (\mc{T}(\mc{B}, \phi), \Phi_t) \simeq \left(\mc{B} \ast \left( \overset{q}{\underset{\frac{1}{1+t}}{\mf{C}}} \oplus \overset{1-q}{\underset{\frac{t}{1+t}}{\mf{C}}} \right) \right)_{\frac{1}{1+t}}.
\]
\end{Prop}

\begin{proof}
Define the map $(\mc{T}(\mc{B}, \phi), \Phi_t) \rightarrow (\mc{A}, \tilde{\phi})$ by $1 \mapsto q$,
\begin{equation}
\label{Eq:map}
X(f) \mapsto (1+t) q f q,
\end{equation}
and extend it as a homomorphism by the universal property of the tensor algebra. Since $\mc{A}$ is a unital algebra with unit $q$, and its elements are linear combinations of $q$ and $q f_1 q \ldots q f_n q$ for $\set{f_i : 1 \leq i \leq n} \subset \mc{B}^\circ$, the map is onto.

By definition, $R^{\Phi_t}[1] = 1 = R^{\tilde{\phi}}[q]$. Using Proposition~\ref{Prop:Free-power} and \cite{Nica-Speicher-Multiplication}, for $f_i \in \mc{B}^\circ$,
\[
R^{\Phi_t}[X(f_1), \ldots, X(f_n)] = (1+t) R^\phi[f_1, \ldots, f_n] = (1+t)^n R^{\tilde{\phi}}[q f_1 q, \ldots, q f_n q].
\]
It follows that the map \eqref{Eq:map} preserves joint distributions. The result now follows from Lemma~\ref{Lemma:vN-Rep}.
\end{proof}

\section{von Neumann algebras}
\label{Sec:vN}

In this section, we investigate the von Neumann algebras generated by $\set{X(f, t) : f \in (\mc{B}^\circ)^{sa}}$ for a fixed $t > 0$ for several examples of *-algebras $\mc{B}$, as well as their subalgebras generated by $\set{X(p_i^\circ, t) : 1 \leq i \leq d}$ for collections of projections $\set{p_i}$ generating $\mc{B}$. Note that in the Gaussian/semiciruclar case, the operators $X(\xi_i)$ corresponding to a basis of the Hilbert space $\mc{H}$ generate the von Neumann algebra of operators $\Gamma(\mc{H})$; but in the case of algebras, operators $X(p_i^\circ)$ corresponding to the generators of the algebra $\mc{B}$ typically generate only a subalgebra of $\Gamma(\mc{B}, \phi)$.

%\subsection{$\mc{B} = \mf{C}^2$}
%\label{Sec:C2}

First we consider the case $\mc{B} = \mf{C}^2$. The following result implies Theorem~\ref{Thm:C2-Intro}.

\begin{Thm}
\label{Thm:C2}
Fix $\alpha \in (0,1)$. Let $\mc{B} = \underset{\alpha}{\overset{p}{\mf{C}}} \oplus \overset{1 - p}{\underset{1 - \alpha}{\mf{C}}}$, that is, $\mc{B} \simeq \mf{C}^2$ is generated by $1$ and a projection $p$ with $\phi(p) = \alpha$. Without loss of generality, we may assume that $\alpha \leq \frac{1}{2}$. $\mc{B}^\circ$ is spanned by $p^\circ = p - \alpha$, and $\mc{B} \simeq \mf{C}^2$ is also generated by $p^\circ$ and $1$. So $\Gamma(\mc{B}, \phi; t)$ is generated by a single element $X(p^\circ)$ and $1$.
\begin{enumerate}
\item
The distribution of $p$ is Bernoulli with parameter $\alpha$.
\item
The distribution of $X(p^\circ, t)$ has absolutely continuous part
\begin{equation}
d\mu_{X(p^\circ, t)}^\Phi = \frac{(t+1)\sqrt{4t\alpha (1-\alpha)  -\big( (1-2\alpha) - x \big)^2}}{2\pi(\alpha(1+t) + x)((1-\alpha)(1+t)-x)}dx,
\end{equation}
with support $[ 1 - 2\alpha - 2\sqrt{t \alpha (1-\alpha)} , 1 - 2\alpha + 2\sqrt{t \alpha (1-\alpha)} ]$ and atoms
\begin{align*}
\mu_{X(p^\circ, t)}^\Phi(\{-\alpha(1+t)\}) &= \max\{ 1-\alpha(1+t) ,0\} \\
\mu_{X(p^\circ, t)}^\Phi(\{(1-\alpha)(1+t)\}) &= \max\{ \alpha(1+t) - t ,0\}
\end{align*}
In particular, for $\alpha = \frac{1}{2}$,
\begin{equation}
d\mu_{X(p^\circ, t)}^\Phi = \frac{(t+1)\sqrt{t - x^2}}{2\pi \left(\left(\frac{1+t}{2} \right)^2 - x^2 \right)} \,dx,
\end{equation}
and
\[
\mu_{X(p^\circ, t)}^\Phi\left(\set{\pm \frac{1+t}{2}}\right) = \max\set{ \frac{1-t}{2}, 0 }.
\]
\item
The $W^\ast$-probability space $W^\ast(\mc{T}(\mc{B}, \phi), \Phi_t)$ is isomorphic to $L^\infty (\mu_{X(p^\circ, t)}^\Phi)$ direct sum one or two copies of $\mf{C}$.
\item
The distribution of $X(p^\circ, t)$ with respect to $\Psi$ has absolutely continuous part
\begin{equation}
d\mu_{X(p^\circ, t)}^\Psi = \frac{\sqrt{4t\alpha (1-\alpha)  -\big( (1-2\alpha) - x \big)^2} }{2 \pi (t\alpha(1-\alpha) + (1-2\alpha)x) }dx,
\end{equation}
with support $[1 - 2\alpha - 2\sqrt{t\alpha(1 - \alpha)}, 1 - 2\alpha + 2\sqrt{t\alpha(1 - \alpha)}]$ and, for $\alpha \neq \frac{1}{2}$, an atom
\begin{equation*}
d\mu_{X(p^\circ, t)}^\Psi\left(\left\{ \frac{-t\alpha(1-\alpha)}{1-2\alpha} \right\}\right) = \max\left\{ 1 -t\frac{\alpha(1-\alpha)}{(1-2\alpha)^2},0 \right\}
\end{equation*}
For $\alpha = \frac{1}{2}$,
\begin{equation}
d\mu_{X(p^\circ, t)}^\Psi = \frac{2}{\pi t} \sqrt{t - x^2} dx.
\end{equation}
\end{enumerate}
\end{Thm}

\begin{proof}
(a) follows from direct computation and observing the moments are $1, \alpha, \alpha, \alpha, \ldots$.

For (b), we will use Proposition~\ref{Prop:Free-power}. The Cauchy transform of the distribution of $p$ is
	\begin{equation*}
G(z) = \frac{1-\alpha}{z} + \frac{\alpha}{z-1},
	\end{equation*}
	whose inverse is
	\begin{equation*}
z = \frac{\eta+1 \pm\sqrt{(\eta+1)^2 - 4\eta(1-\alpha)}}{2\eta}.
	\end{equation*}
Subtracting $\frac{2}{2\eta}$ from this gives the R-transform. To apply Proposition \ref{Prop:Free-power} requires the distribution of $p-\alpha$, whose R-transform is
\begin{equation*}
z = \frac{(1-2\alpha)\eta-1 \pm\sqrt{(\eta+1)^2 - 4\eta(1-\alpha)}}{2\eta}.
\end{equation*}
The R-transform of its $(1+t)$th convolution power is
\begin{equation*}
z = (1+t)\frac{(1-2\alpha)\eta-1 \pm\sqrt{(\eta+1)^2 - 4\eta(1-\alpha)}}{2\eta}.
\end{equation*}
Adding $\frac{2}{2\eta}$ to this and inverting gives the corresponding Cauchy transform
\begin{equation*}
G_{1+t}(z) = \frac{(1-2\alpha)(1+t) - (1-t)z - (t+1)\sqrt{ 4\alpha^2 (1+t) - 4\alpha(1+t-z) + (z-1)^2 }}{2(\alpha(1+t) + z)((1-\alpha)(1+t)-z)}.
\end{equation*}
Applying Stieltjes inversion gives the result. Applying Proposition 8 from Chapter 3 of \cite{Mingo-Speicher-book} gives the atoms. (c) follows from (b).

(d) Represent $\mc{B}$ as a subalgebra of $M_2(\mf{C})$, with $\phi$ the vector state of $(1,0)$, and $p \in M_2(\mf{C})$ a projection such that $p_{11} = \alpha$. Then $p$ must have the form $\begin{pmatrix}
\alpha & \beta \\
\beta & 1- \alpha
	\end{pmatrix}$ where $\beta^2 = \alpha(1-\alpha)$, so $p^\circ = \begin{pmatrix}
	0 & \beta \\
	\beta & 1- 2\alpha
	\end{pmatrix}$
Therefore by Proposition~\ref{Prop:TwoStateFCs},
	\begin{align*}
R_n^\Psi [X(p^\circ, t)] &= tB_n^\phi[p^\circ] \\
&= t\langle p^\circ\Lambda(p^\circ, \Lambda(p^\circ, \ldots, \Lambda(p^\circ,p^\circ)\ldots))\Omega, \Omega\rangle \\
&= t\alpha(1-\alpha)(1-2\alpha)^{n-2},
	\end{align*}
	and so the R-transform with respect to $\Psi$ is the geometric series
	\begin{equation*}
t\alpha(1-\alpha)\sum_{n=1}^\infty (1-2\alpha)^{n-1}z^n = \frac{t\alpha(1-\alpha)z}{1-(1-2\alpha)z}.
	\end{equation*}
Add $\frac{1}{z}$ to this and invert to get the Cauchy transform
\begin{equation*}
z = \frac{1 - 2\alpha + \eta \pm \sqrt{(1-2\alpha + \eta)^2 - 4(t\alpha(1-\alpha) + (1-2\alpha)\eta)}}{2(t\alpha(1-\alpha)+(1-2\alpha)\eta)}.
\end{equation*}
Apply Stieltjes inversion to get the absolutely continuous part. We have an atom at $x = \frac{-t\alpha(1-\alpha)}{1-2\alpha}$, with measure $\max\left\{ 1 -t\frac{\alpha(1-\alpha)}{(1-2\alpha)^2},0 \right\}$.
\end{proof}

\begin{comment}
\begin{Cor}
\label{Cor:Rescaling}
Fix $\alpha = \frac{1}{2}$. Re-scale the time by $t = \frac{\theta}{1 - \theta}$ and the variable itself by $2 \sqrt{1 - \theta}$. The distribution of $2 \sqrt{1 - \theta} X(p^\circ)$ with respect to $\Phi_t$ is
\[
\frac{\sqrt{4 \theta - x^2}}{2 \pi (1 - (1 - \theta) x^2)} \,dx + \max \left( \frac{1 - 2 \theta}{2 (1 -\theta)}, 0 \right) (\delta_{-1/\sqrt{1 - \theta}} + \delta_{1/ \sqrt{1 - \theta}}),
\]
and with respect to $\Psi_t$ it is
\[
\frac{1}{2 \pi \theta} \sqrt{4 \theta - x^2} \,dx.
\]
\end{Cor}
\end{comment}

%\subsection{$\mc{B} = \ast_{i=1}^d \mf{C}^2$ (Free product of \ref{Sec:C2})}
%\label{Sec:Free-product}

\begin{Thm}
\label{Thm:Free-product}
Let $\mc{B}$ be the free product $\mc{B} = \ast_{i=1}^d \left( \underset{\alpha_i}{\overset{p_i}{\mf{C}}} \oplus \overset{1 - p_i}{\underset{1 - \alpha_i}{\mf{C}}} \right)$, the $i$'th copy generated by the identity and the projection $p_i$ of state $\alpha_i \leq \frac{1}{2}$. In $W^\ast(\mc{T}(\mc{B}, \phi), \Phi_t)$, the von Neumann subalgebra generated by $X(p_i^\circ)$ is
\[
W^\ast(X(p_i^\circ) : 1 \leq i \leq d)  \simeq
\begin{cases}
\underset{1 - \gamma_1 - \gamma_2}{L(\mf{F}_x)} \oplus \underset{\gamma_1}{\mf{C}} \oplus \underset{\gamma_2}{\mf{C}}, & \quad t < \frac{\alpha_d - \sum_{i=1}^{d-1}\alpha_i}{1 - \left( \alpha_d - \sum_{i=1}^{d-1}\alpha_i \right)}, \\
\underset{1 - \gamma_1}{L(\mf{F}_x)} \oplus \underset{\gamma_1}{\mf{C}}, & \quad \frac{\alpha_d - \sum_{i=1}^{d-1}\alpha_i}{1 - \left( \alpha_d - \sum_{i=1}^{d-1}\alpha_i \right)} \leq t < \frac{1 - \left( \left( \sum_{i=1}^d\alpha_i \right) \right)}{ \left( \sum_{i=1}^d\alpha_i \right) }, \\
\mf{F}_d, & \quad \frac{1 - \left( \left( \sum_{i=1}^d\alpha_i \right) \right)}{ \left( \sum_{i=1}^d\alpha_i \right)} \leq t,
\end{cases}
\]
where
	\begin{align*}
		\gamma_1 &= \max\left\{ 1 - \left( \sum_{i=1}^{d}\alpha_1 \right)(1+t), 0\right\} \\
		\gamma_2 &= \max\left\{ \left( \alpha_d - \sum_{i=1}^{d-1}\alpha_i \right)(1+t) -t, 0 \right\}
	\end{align*}
	and $x$ is chosen so that the free dimension is the sum of the free dimensions of $W^\ast(X(p_i^\circ))$.
\end{Thm}

\begin{proof}
Since $\alpha_d \leq \frac{1}{2}$,
\[
\frac{\alpha_d - \sum_{i=1}^{d-1}\alpha_i}{1 - \left( \alpha_d - \sum_{i=1}^{d-1}\alpha_i \right)} \leq \frac{1 - \left( \left( \sum_{i=1}^d\alpha_i \right) \right)}{ \left( \sum_{i=1}^d\alpha_i \right) }.
\]
So the statement of the theorem is equivalent to $W^\ast(X(p_i^\circ) : 1 \leq i \leq d)  \simeq \underset{1 - \gamma_1 - \gamma_2}{L(\mf{F}_x)} \oplus \underset{\gamma_1}{\mf{C}} \oplus \underset{\gamma_2}{\mf{C}}$, for $\gamma_1, \gamma_2$ as above. We will prove this by induction on $d$. For $d=1$, this follows from Theorem~\ref{Thm:C2}. Since by assumption, $\set{p_i^\circ : 1 \leq i \leq d}$ are free in $(\mc{B}, \phi)$, by Corollary~\ref{Cor:Independence-functor},
\[
W^\ast(X(p_i^\circ) : 1 \leq i \leq d) = \ast_{i=1}^d W^\ast(X(p_i^\circ)).
\]
Since free product is commutative, without loss of generality we may assume that $\alpha_i$'s are increasing.

Suppose the statement holds for $d$. Then by Theorem 2.4 of \cite{Dykema-Free-products-hyperfinite},
	\begin{equation*}
		\ast_{i=1}^{d+1} W^\ast(X(p_i^\circ)) \simeq \underset{1-\gamma_{11}-\gamma_{12} - \gamma_{21} - \gamma_{22}}{L(\mf{F}_x)} \oplus \underset{\gamma_{11}}{\mf{C}}\oplus \underset{\gamma_{12}}{\mf{C}} \oplus \underset{\gamma_{21}}{\mf{C}} \oplus \underset{\gamma_{22}}{\mf{C}} ,
	\end{equation*}
	where
	\begin{align*}
		\gamma_{11} &= \max\left\{ 1 -  \left( \sum_{i=1}^{d+1}\alpha_i \right)(1+t), 0 \right\} \\
		\gamma_{12} &= \max\left\{ \left( \alpha_{d} - \alpha_{d+1} - \sum_{i=1}^{d-1}\alpha_i \right)(1+t) - t, 0\right\} \\
		\gamma_{21} &= \max\left\{ \left( \alpha_{d+1} - \sum_{i=1}^{d}\alpha_i \right)(1+t) - t, 0 \right\} \\
		\gamma_{22} &= \max\left\{ \left( \alpha_{d+1} + \alpha_{d} - \sum_{i=1}^{d-1}\alpha_i \right)(1+t) - 2t - 1, 0 \right\}.
	\end{align*}
Note that this expansion holds even if some of $\gamma_1, \gamma_2$ are zero. Since $\alpha_{d} \leq \alpha_{d+1}$, $\gamma_{12} = 0$. Since $\alpha_{d}, \alpha_{d+1} \leq \frac{1}{2}$, $\gamma_{22} = 0$. Finally, $\gamma_{11}$ and $\gamma_{21}$ are precisely the forms of $\gamma_1$, $\gamma_2$ for $d+1$.
The result follows.
\end{proof}

%\subsection{$\mc{B} = M_{d+1}(\mf{C})$, $d \geq 2$}

Next, we consider the case $\mc{B} = M_{d+1}(\mf{C})$, with a vector state, for $d \geq 2$.

\begin{Remark}
Following \cite{Arizmendi-Torres-c-free-Hilbert}, for $1 \leq i \leq d$, let $\mc{H}_i = \mf{C} \Omega_i \oplus \mc{H}_i^\circ$ be pointed Hilbert spaces, and $(\mc{B}_i, \phi_i)$ be $\ast$-probability spaces represented on them as in Construction~\ref{Constr:Centered}. Let
\[
\mc{H} = \mf{C} \Omega \oplus \bigoplus_{i=1}^d \mc{H}_i^\circ.
\]
Represent each $\mc{B}_i$ on $\mc{H}$ by
\[
\begin{split}
b_i (s \Omega \oplus \xi_1 \oplus \ldots \oplus \xi_d)
& = s \phi_i[b_i] \Omega \oplus 0 \oplus \ldots \oplus (s b_i \Omega - s \phi_i[b_i] \Omega) \oplus \ldots \oplus 0 \\
&\quad + \ip{b_i \xi_i}{\Omega_i} \Omega \oplus 0 \oplus \ldots \oplus (b_i \xi_i - \ip{b_i \xi_i}{\Omega_i} \Omega) \oplus \ldots \oplus 0.
\end{split}
\]
Let $\mc{B}$ be the algebra generated by these non-unital embeddings of $\set{\mc{B}_i : 1 \leq i \leq d}$ in $\mf{B}(\mc{H})$, and $\phi$ the vector state given by $\Omega$. Then $\set{\mc{B}_i : 1 \leq i \leq d}$ are Boolean independent in $(\mc{B}, \phi)$.
%\cred{If each representation of $\mc{B}_i$ on $\mc{H}_i$ is faithful, so are their representations on $\mc{H}$.}
\end{Remark}

\begin{Cor}
\label{Cor:Boolean-product-C2}
For $d \geq 2$ and $\set{\alpha_i : 1 \leq i \leq d} \subset (0,1)$, the Boolean product of algebras $\underset{\alpha_i}{\overset{p_i}{\mf{C}}} \oplus \overset{1 - p_i}{\underset{1 - \alpha_i}{\mf{C}}}$ for $1 \leq i \leq d$ is isomorphic to $M_{d+1}(\mf{C})$ with a vector state.
\end{Cor}

\begin{proof}
Let $\mc{H} = \mf{C} \Omega \oplus \bigoplus_{i=1}^d \mc{H}_i$, where each $\mc{H}_i \simeq \mf{C}$. On $\mf{B}(\mf{C} \Omega \oplus \mc{H}_i) = M_2(\mf{C})$, let $\phi_i$ be the vector state for $\Omega$, in other words the $(1,1)$ entry of the matrix. In this algebra, let $p_i$ be a projection with $\phi[p_i] = \alpha_i$, and $p^\circ_i = p_i - \alpha_i$. Denote by $\mc{B}$ the subalgebra of $M_{d+1}(\mf{C})$ generated by $\set{p^\circ_i: 1 \leq i \leq d}$.

Note that $p^\circ_i = \beta_i (E_{1, i+1} + E_{i+1, 1}) + \gamma_i E_{i+1, i+1}$, with $\beta_i \neq 0$. Then for $i \neq j$, $p^\circ_i p^\circ_j = \beta_i \beta_j E_{i+1, j+1}$. Therefore $E_{i+1, j+1} \in \mc{B}$ for $i \neq j$. Multiplying these, also $E_{i+1, i+1} \in \mc{B}$. Next, $p^\circ_i E_{i+1, i+1} = \beta_i E_{1, i+1} + \gamma_i E_{i+1, i+1}$, so also $E_{1, i+1} \in \mc{B}$, as is $E_{i+1, 1}$. Multiplying these, also $E_{1,1} \in \mc{B}$.
\end{proof}

\begin{Prop}
\label{Prop:Boolean-Psi}
Let $\set{\alpha_i : 1 \leq i \leq d} \subset (0,1)$, and $\set{p_i : 1 \leq i \leq d}$ as in the preceding Corollary. In $W^\ast(\mc{T}(M_{d+1}(\mf{C}), \phi_{11}), \Psi_t)$, the von Neumann subalgebra generated by $X(p_i^\circ)$ is
\[
W^\ast(X(p_i^\circ) : 1 \leq i \leq d) \simeq \underset{1-\gamma}{L(\mf{F}_x)}\oplus \underset{\gamma}{\mf{C}},
\]
where
\begin{equation*}
\gamma = \max\left\{ 1 - t\sum_{i=1}^d \frac{\alpha_i (1-\alpha_i)}{(1-2\alpha_i)^2}, 0\right\},
\end{equation*}
with the convention that $\gamma = 0$ if any $\alpha_i = \frac{1}{2}$. $x$ is chosen so that the free dimension is the sum of the free dimensions of $W^\ast(X(p_i^\circ))$.
\end{Prop}

\begin{proof}
The argument is similar to Theorem~\ref{Thm:Free-product}. By Corollary~\ref{Cor:Boolean-product-C2}, $\set{p_i^\circ : 1 \leq i \leq d}$ are Boolean independent in $(M_{d+1}(\mf{C}), \phi_{11})$. Therefore by Corollary~\ref{Cor:Independence-functor}(b),
\[
W^\ast(X(p_i^\circ) : 1 \leq i \leq d) = \ast_{i=1}^d W^\ast(X(p_i^\circ)).
\]
For $d=1$, it follows from Theorem~\ref{Thm:C2} that $W^\ast(X(p_1^\circ)) \simeq L(\mf{F}_1)$ for $\alpha_1 = \frac{1}{2}$, and otherwise
\[
W^\ast(X(p_1^\circ)) \simeq \underset{1-\gamma_1}{L(\mf{F}_1)}\oplus \underset{\gamma_1}{\mf{C}},
\]
where
\begin{equation*}
\gamma_1 = \max\left\{ 1 - t \frac{\alpha_1 (1-\alpha_1)}{(1-2\alpha_1)^2}, 0\right\},
\end{equation*}
In this setting, Theorem 2.4 of \cite{Dykema-Free-products-hyperfinite} states that
\[
\ast_{i=1}^d \left( \underset{1-\gamma_i}{L(\mf{F}_1)}\oplus \underset{\gamma_i}{\mf{C}} \right) \simeq \underset{1-\gamma}{L(\mf{F}_x)}\oplus \underset{\gamma}{\mf{C}},
\]
where
\begin{equation*}
\gamma = \max \left( \sum_{i=1}^d \gamma_i - (d-1), 0 \right).
\end{equation*}
The result follows.
\end{proof}

\begin{proof}[Proof of Theorem~\ref{Thm:Boolean-Phi}]
In the setting of Construction~\ref{Constr:Ricard}, Ricard \cite{Ricard-t-Gaussian} proved that with respect to the vacuum state $\phi$, $Z(f, \theta)$ has the distribution
\begin{equation}
\label{Eq:Ricard1}
\frac{\sqrt{4 \theta - x^2}}{2 \pi (1 - (1 - \theta) x^2)} \,dx + \max \left( \frac{1 - 2 \theta}{2 (1 -\theta)}, 0 \right) (\delta_{-1/\sqrt{1 - \theta}} + \delta_{1/ \sqrt{1 - \theta}}),
\end{equation}
while with respect to another state $\psi$ it has the distribution
\begin{equation}
\label{Eq:Ricard2}
\frac{1}{2 \pi \theta} \sqrt{4 \theta - x^2} \,dx.
\end{equation}
Moreover, Ricard proved that for orthogonal $f_i$'s, these operators are conditionally free with respect to the pair $(\phi, \psi)$.

By (the proof of) Corollary~\ref{Cor:Boolean-product-C2}, $\set{p_i^\circ : 1 \leq i \leq d}$ are Boolean independent in $(M_{d+1}(\mf{C}), \phi_{11})$. Therefore by Corollary~\ref{Cor:Independence-functor}(b), $\set{X(p_i^\circ) : 1 \leq i \leq d}$ are conditionally free with respect to the pair $(\Phi_t, \Psi_t)$. Moreover, the choice of $p_i$ in the hypothesis of this theorem corresponds to $\alpha = \frac{1}{2}$ in Theorem~\ref{Thm:C2}. Re-scale the time by $t = \frac{\theta}{1 - \theta}$. Then according to that theorem, the distributions of each $2 \sqrt{1 - \theta} X(p_i^\circ)$ with respect to $\Phi_t$ and $\Psi_t$ match precisely the distributions in \eqref{Eq:Ricard1} and \eqref{Eq:Ricard2} above. It follows that the von Neumann algebras generated by $\set{Z(f_i, \theta) : 1 \leq i \leq d}$ and $\set{X(p_i^\circ) : 1 \leq i \leq d}$ are isomorphic. Ricard showed that this von Neumann algebra is
\begin{equation}
\label{Eq:Ricard-vN}
\Gamma_{\theta,d} = \begin{cases}
L(\mf{F}_d) & \text{ if } \theta \in \left[ \frac{d}{d + \sqrt{d}}, \frac{d}{d - \sqrt{d}} \right]\\
L(\mf{F}_d) \oplus \mf{B}(\ell_2) & \text{ otherwise},
\end{cases}
\end{equation}
and $\Gamma_{\theta, \infty} = \mf{B}(\ell^2)$. Finally, a simple calculation converts the cutoff in \eqref{Eq:Ricard-vN} from $\theta$ to $t$.
\end{proof}

%\subsection{$\mc{B} = L^\infty[0, 2 \pi]$}

Finally, we take $\mc{B} = L^\infty[0, 2 \pi]$ with Lebesgue measure.

\begin{proof}[Proof of Theorem~\ref{Thm:L-infty}]
Applying Lemma 2.2 in \cite{Dykema-Free-products-hyperfinite} with either $\beta$ or $1-\beta$ equal to $\frac{1}{1 + t}$ so that $\beta \geq \frac{1}{2}$,
\[
L^\infty[0,1] \ast \left( \underset{\beta}{\overset{q}{\mf{C}}} \oplus \overset{1 - q}{\underset{1 - \beta}{\mf{C}}} \right) \simeq L\left(\mf{F}\left(1 + 2 \beta (1 - \beta)\right)\right) = L\left(\mf{F}\left(1 + \frac{2 t}{(1+t)^2}\right)\right),
\]
Therefore by Proposition~\ref{Prop:Compression},
\[
W^\ast(\mc{T}(L^\infty[0,1], dx), \Phi_t) \simeq L\left(\mf{F}\left(1 + \frac{2 t}{(1+t)^2}\right)\right)_{\frac{1}{1+t}} \simeq L(\mf{F}_{1 + 2 t}). \qedhere
\]
\end{proof}

The comparison between Theorem~\ref{Thm:L-infty} and the following proposition explains the choice of the coefficients in the annihilation operator in the main construction.

\begin{Prop}
\label{Prop:s-neq-t}
Define the operators $a^+$, $a^0$, $a^-$ as in the main construction, with the exception that $a^-(f) (g) = (1 + s) a^-_\phi(f) (g)$. Suppose that $s \neq t$.
\begin{enumerate}
\item
The vacuum state is in general not tracial.
\item
Suppose in addition that $\mc{B} = L^\infty[0, 1)$, with $\phi$ the Lebesgue integration. Then the $(s \neq t)$ version of $\Gamma_w(\mc{B}, \phi; t)$ is the algebra of all bounded operators on the Fock space.
\end{enumerate}
\end{Prop}

\begin{proof}
The argument for part (a) is the same as in Lemma~\ref{Lemma:Nontracial}. For part (b), the proof follows the general idea of \cite{Wysoczanski-infinite-t}. Let $e_k = e^{2\pi i k \theta}$. We first show that
\[
\frac{1}{N}\sum_{k=1}^{N} X(e_k)X(e_k^\ast) \rightarrow S:= (1+t)I + (s-t)P_\Omega
\]
in SOT, where $P_\Omega$ is the projection onto $\Omega$. Decompose each $X(e_k) = a^+(e_k) + a^-(e_k) + a^0(e_k)$. Then we have $9$ cases. It is easy to check that
\[
a^-(e_k) a^+(e_k^\ast) = t (I - P_\Omega) + (1 + s) P_\Omega, \quad a^0(e_k) a^0(e_k^\ast) = I - P_\Omega,
\]
and
\[
\frac{1}{N}\sum_{k=1}^{N} a^\sharp(e_k) a^\flat(e_k^\ast) \rightarrow 0
\]
for all the other choices of $\sharp, \flat$. It follows that $S$ is in $\Gamma_{w}(\mc{B},\phi;t)$. Next,
\[
S^2 - (1+t)S = (1+s)(s-t)P_\Omega,
\]
which shows $P_\Omega\in \Gamma_{w}(\mc{B},\phi;t)$ since $s\neq t$. It is also easy to check that $\Omega$ is cyclic for $\Gamma_{w}(\mc{B},\phi;t)$. So by Proposition 2.4 of \cite{Wysoczanski-infinite-t}, the commutant of $\Gamma_{w}(\mc{B},\phi;t)$ is trivial. Thus $\Gamma_{w}(\mc{B},\phi;t)$ is the algebra of all bounded operators on the Fock space.
\end{proof}

\def\cprime{$'$} \def\cprime{$'$}
\providecommand{\bysame}{\leavevmode\hbox to3em{\hrulefill}\thinspace}
\providecommand{\MR}{\relax\ifhmode\unskip\space\fi MR }
% \MRhref is called by the amsart/book/proc definition of \MR.
\providecommand{\MRhref}[2]{%
  \href{http://www.ams.org/mathscinet-getitem?mr=#1}{#2}
}
\providecommand{\href}[2]{#2}

%\bibliographystyle{amsalpha}
%\bibliography{bibdata}

\end{document}